\documentclass[11pt]{amsart}
\usepackage{amsmath, amsthm, amscd, amsfonts, amssymb, graphicx, color,float,pgf,tikz}
\usepackage{amssymb,fontenc}
\usepackage{latexsym,wasysym,mathrsfs}
\usepackage{hyperref}
\usetikzlibrary{arrows}
\usepackage[square,numbers,sort&compress]{natbib}

\makeatletter \oddsidemargin.9375in \evensidemargin \oddsidemargin
\newtheorem{theorem}{Theorem}[section]
\newtheorem{lemma}[theorem]{Lemma}

\newtheorem{corollary}[theorem]{Corollary}

\theoremstyle{definition}
\newtheorem{definition}[theorem]{Definition}
\newtheorem{example}[theorem]{Example}

\newtheorem{remark}[theorem]{Remark}
\numberwithin{equation}{section}

\newcommand{\be}{\begin{equation}}
\newcommand{\ee}{\end{equation}}

\numberwithin{equation}{section}

\usepackage{etoolbox}
\makeatletter
\patchcmd{\@settitle}{\uppercasenonmath\@title}{}{}{}
\patchcmd{\@setauthors}{\MakeUppercase}{}{}{}
\makeatother

\allowdisplaybreaks
\begin{document}

\title[Continuous $\ast$-K-g-Frame in Hilbert $C^{\ast}$-Modules]{Continuous $\ast$-K-g-Frame in Hilbert $C^{\ast}$-Modules}

\author[A. TOURI$^{*}$, M. ROSSAFI, H. LABRIGUI, A. AKHLIDJ]{A. TOURI$^1$, M. ROSSAFI$^1$, H. LABRIGUI $^1$ \MakeLowercase{and} A. AKHLIDJ$^2$}

\address{$^{1}$Department of Mathematics, University of Ibn Tofail, B.P. 133, Kenitra, Morocco}
\email{\textcolor[rgb]{0.00,0.00,0.84}{rossafimohamed@gmail.com;  hlabrigui75@gmail; touri.abdo68@gmail.com}}
\address{$^2$Department of Mathematics,
	University of Hassan II,
	Casablanca
	Morocco}
\email{\textcolor[rgb]{0.00,0.00,0.84}{akhlidjabdellatif@gmail.com}}




\keywords{Continuous Frame, Continuous $\ast$-K-g-frame, $C^{\ast}$-algebra, Hilbert $\mathcal{A}$-modules.\\
\indent $^{*}$ Corresponding author}
\maketitle

\begin{abstract}
	In this paper, we introduce the concept of Continuous $\ast$-K-g-Frame in Hilbert $C^{\ast}$-Modules and we give some properties.
\end{abstract}
\maketitle
\vspace{0.1in}

\section{\textbf{Introduction and preliminaries}}
The concept of frames in Hilbert spaces has been introduced by
Duffin and Schaeffer \cite{Duf} in 1952 to study some deep problems in nonharmonic Fourier
series, after the fundamental paper \cite{13} by Daubechies, Grossman and Meyer, frame
theory began to be widely used, particularly in the more specialized context of wavelet
frames and Gabor frames \cite{Gab}.

Traditionally, frames have been used in signal processing, image processing, data compression
and sampling theory. A discreet frame is a countable family of
elements in a separable Hilbert space which allows for a stable, not necessarily unique,
decomposition of an arbitrary element into an expansion of the frame elements. The
concept of a generalization of frames to a family indexed by some locally compact space
endowed with a Radon measure was proposed by G. Kaiser \cite{15} and independently by Ali,
Antoine and Gazeau \cite{11}. These frames are known as continuous frames. Gabardo and
Han in \cite{14} called these frames associated with measurable spaces, Askari-Hemmat,
Dehghan and Radjabalipour in \cite{12} called them generalized frames and in mathematical
physics they are referred to as coherent states \cite{11}. 

In this paper, we introduce the notion of Continuous $\ast$-K-g-Frame which are generalization of $\ast$-K-g-Frame in Hilbert $C^{\ast}$-Modules introduced by M. Rossafi and S. Kabbaj \cite{Ross} and we establish some new results.

The paper is organized as follows, we continue this introductory section we briefly recall the definitions and basic properties of $C^{\ast}$-algebra, Hilbert $C^{\ast}$-modules. In Section 2, we introduce the Continuous $\ast$-K-g-Frame, the Continuous pre-$\ast$-K-g-frame operator and the Continuous $\ast$-K-g-frame operator, also we establish here properties.

In the following we briefly recall the definitions and basic properties of $C^{\ast}$-algebra, Hilbert $\mathcal{A}$-modules. Our reference for $C^{\ast}$-algebras is \cite{{Dav},{Con}}. For a $C^{\ast}$-algebra $\mathcal{A}$ if $a\in\mathcal{A}$ is positive we write $a\geq 0$ and $\mathcal{A}^{+}$ denotes the set of positive elements of $\mathcal{A}$.
\begin{definition}\cite{Pas}.
	
	Let $ \mathcal{A} $ be a unital $C^{\ast}$-algebra and $\mathcal{H}$ be a left $ \mathcal{A} $-module, such that the linear structures of $\mathcal{A}$ and $ \mathcal{H} $ are compatible. $\mathcal{H}$ is a pre-Hilbert $\mathcal{A}$-module if $\mathcal{H}$ is equipped with an $\mathcal{A}$-valued inner product $\langle.,.\rangle_{\mathcal{A}} :\mathcal{H}\times\mathcal{H}\rightarrow\mathcal{A}$, such that is sesquilinear, positive definite and respects the module action. In the other words,
	\begin{itemize}
		\item [(i)] $ \langle x,x\rangle_{\mathcal{A}}\geq0 $ for all $ x\in\mathcal{H} $ and $ \langle x,x\rangle_{\mathcal{A}}=0$ if and only if $x=0$.
		\item [(ii)] $\langle ax+y,z\rangle_{\mathcal{A}}=a\langle x,y\rangle_{\mathcal{A}}+\langle y,z\rangle_{\mathcal{A}}$ for all $a\in\mathcal{A}$ and $x,y,z\in\mathcal{H}$.
		\item[(iii)] $ \langle x,y\rangle_{\mathcal{A}}=\langle y,x\rangle_{\mathcal{A}}^{\ast} $ for all $x,y\in\mathcal{H}$.
	\end{itemize}	 
	For $x\in\mathcal{H}, $ we define $||x||=||\langle x,x\rangle_{\mathcal{A}}||^{\frac{1}{2}}$. If $\mathcal{H}$ is complete with $||.||$, it is called a Hilbert $\mathcal{A}$-module or a Hilbert $C^{\ast}$-module over $\mathcal{A}$. For every $a$ in $C^{\ast}$-algebra $\mathcal{A}$, we have $|a|=(a^{\ast}a)^{\frac{1}{2}}$ and the $\mathcal{A}$-valued norm on $\mathcal{H}$ is defined by $|x|=\langle x, x\rangle_{\mathcal{A}}^{\frac{1}{2}}$ for $x\in\mathcal{H}$.
	
	Let $\mathcal{H}$ and $\mathcal{K}$ be two Hilbert $\mathcal{A}$-modules, A map $T:\mathcal{H}\rightarrow\mathcal{K}$ is said to be adjointable if there exists a map $T^{\ast}:\mathcal{K}\rightarrow\mathcal{H}$ such that $\langle Tx,y\rangle_{\mathcal{A}}=\langle x,T^{\ast}y\rangle_{\mathcal{A}}$ for all $x\in\mathcal{H}$ and $y\in\mathcal{K}$.
	
We reserve the notation $End_{\mathcal{A}}^{\ast}(\mathcal{H},\mathcal{K})$ for the set of all adjointable operators from $\mathcal{H}$ to $\mathcal{K}$ and $End_{\mathcal{A}}^{\ast}(\mathcal{H},\mathcal{H})$ is abbreviated to $End_{\mathcal{A}}^{\ast}(\mathcal{H})$.
\end{definition}

The following lemmas will be used to prove our mains results
\begin{lemma} \label{1} \cite{Pas}.
	Let $\mathcal{H}$ be Hilbert $\mathcal{A}$-module. If $T\in End_{\mathcal{A}}^{\ast}(\mathcal{H})$, then $$\langle Tx,Tx\rangle\leq\|T\|^{2}\langle x,x\rangle, \forall x\in\mathcal{H}.$$
\end{lemma}

\begin{lemma} \label{sb} \cite{Ara}.
	Let $\mathcal{H}$ and $\mathcal{K}$ two Hilbert $\mathcal{A}$-modules and $T\in End^{\ast}(\mathcal{H},\mathcal{K})$. Then the following statements are equivalent:
	\begin{itemize}
		\item [(i)] $T$ is surjective.
		\item [(ii)] $T^{\ast}$ is bounded below with respect to norm, i.e., there is $m>0$ such that $\|T^{\ast}x\|\geq m\|x\|$ for all $x\in\mathcal{K}$.
		\item [(iii)] $T^{\ast}$ is bounded below with respect to the inner product, i.e., there is $m'>0$ such that $\langle T^{\ast}x,T^{\ast}x\rangle\geq m'\langle x,x\rangle$ for all $x\in\mathcal{K}$.
	\end{itemize}
\end{lemma}
\begin{lemma} \label{3} \cite{Deh}.
	Let $\mathcal{H}$ and $\mathcal{K}$ two Hilbert $\mathcal{A}$-modules and $T\in End^{\ast}(\mathcal{H},\mathcal{K})$. Then:
	\begin{itemize}
		\item [(i)] If $T$ is injective and $T$ has closed range, then the adjointable map $T^{\ast}T$ is invertible and $$\|(T^{\ast}T)^{-1}\|^{-1}\leq T^{\ast}T\leq\|T\|^{2}.$$
		\item  [(ii)]	If $T$ is surjective, then the adjointable map $TT^{\ast}$ is invertible and $$\|(TT^{\ast})^{-1}\|^{-1}\leq TT^{\ast}\leq\|T\|^{2}.$$
	\end{itemize}	
\end{lemma}

\section{\textbf{Continuous $\ast$-K-g-Frame in Hilbert $C^{\ast}$-Modules}}
Let $X$ be a Banach space, $(\Omega,\mu)$ a measure space, and function $f:\Omega\to X$ a measurable function. Integral of the Banach-valued function $f$ has defined Bochner and others. Most properties of this integral are similar to those of the integral of real-valued functions. Because every $C^{\ast}$-algebra and Hilbert $C^{\ast}$-module is a Banach space thus we can use this integral and its properties.

Let $(\Omega,\mu)$ be a measure space, let $U$ and $V$ be two Hilbert $C^{\ast}$-modules, $\{V_{w}: w\in\Omega\}$  is a sequence of subspaces of V, and $End_{\mathcal{A}}^{\ast}(U,V_{w})$ is the collection of all adjointable $\mathcal{A}$-linear maps from $U$ into $V_{w}$.
We define
\begin{equation*}
	\oplus_{w\in\Omega}V_{w}=\left\{x=\{x_{w}\}: x_{w}\in V_{w}, \left\|\int_{\Omega}|x_{w}|^{2}d\mu(w)\right\|<\infty\right\}.
\end{equation*}
For any $x=\{x_{w}: w\in\Omega\}$ and $y=\{y_{w}: w\in\Omega\}$, if the $\mathcal{A}$-valued inner product is defined by $\langle x,y\rangle=\int_{\Omega}\langle x_{w},y_{w}\rangle d\mu(w)$, the norm is defined by $\|x\|=\|\langle x,x\rangle\|^{\frac{1}{2}}$, the $\oplus_{w\in\Omega}V_{w}$ is a Hilbert $C^{\ast}$-module.
\begin{definition}\textcolor{white}{.}
		
	Let $K\in  End_{\mathcal{A}}^{\ast}(U)$,
	We call $\{\Lambda_{w}\in End_{\mathcal{A}}^{\ast}(U,V_{w}): w\in\Omega\}$ a continuous $\ast$-K-g-frame for Hilbert $C^{\ast}$-module $U$ with respect to $\{V_{w}: w\in\Omega\}$ if:
	\begin{itemize}
		\item for any $x\in U$, the function $\tilde{x}:\Omega\rightarrow V_{w}$ defined by $\tilde{x}(w)=\Lambda_{w}x$ is measurable;
		\item there exist two strictly nonzero elements $A$ and $B$ in $\mathcal{A}$ such that
		\begin{equation} \label{2.1}
			A\langle K^{\ast}x,K^{\ast}x\rangle A^{\ast}\leq\int_{\Omega}\langle\Lambda_{w}x,\Lambda_{w}x\rangle d\mu(w)\leq B\langle x,x\rangle B^{\ast}, \forall x\in U.
		\end{equation}
	\end{itemize}
	The elements $A$ and $B$ are called continuous $\ast$-K-g-frame bounds. 
	
	If $A=B$ we call this continuous $\ast$-K-g-frame a continuous tight $\ast$-K-g-frame, and if $A=B=1_{\mathcal{A}}$ it is called a continuous Parseval $\ast$-K-g-frame. If only the right-hand inequality of \eqref{2.1} is satisfied, we call $\{\Lambda_{w}: w\in\Omega\}$ a 
	continuous $\ast$-K-g-Bessel for $U$ with respect to $\{\Lambda_{w}: w\in\Omega\}$ with Bessel bound $B$.
	
\end{definition}
\begin{example}\textcolor{white}{.}
	Let $l^{\infty}$ be the set of all bounded complex-valued sequences. For any $u=\{u_{j}\}_{j\in\mathbf{N}}, v=\{v_{j}\}_{j\in\mathbf{N}}\in l^{\infty}$, we define$$
	uv=\{u_{j}v_{j}\}_{j\in\mathbf{N}}, u^{\ast}=\{\bar{u_{j}}\}_{j\in\mathbf{N}}, \|u\|=\sup_{j\in\mathbf{N}}|u_{j}|.$$
	Then $\mathcal{A}=\{l^{\infty}, \|.\|\}$ is a $\mathbb{C}^{\ast}$-algebra.
	
	Let $\mathcal{H}=C_{0}$ be the set of all sequences converging to zero. For any $u, v\in\mathcal{H}$ we define$$\langle u,v\rangle=uv^{\ast}=\{u_{j}\bar{u_{j}}\}_{j\in\mathbf{N}}.$$
	Then $\mathcal{H}$ is a Hilbert $\mathcal{A}$-module.
	
	Define $f_{j}=\{f_{i}^{j}\}_{i\in\mathbf{N}^{\ast}}$ by $f_{i}^{j}=\frac{1}{2}+\frac{1}{i}$ if $i=j$ and $f_{i}^{j}=0$ if $i\neq j$ $\forall j\in\mathbf{N}^{\ast}$.

Now define the adjointable operator $\Lambda_{j}: \mathcal{H}\to\mathcal{A},\;\; \Lambda_{j}x=\langle x,f_{j}\rangle$.

then for every $x\in\mathcal{H}$ we have$$\sum_{j\in\mathbf{N}}\langle\Lambda_{j}x,\Lambda_{j}x\rangle=\{\frac{1}{2}+\frac{1}{i}\}_{i\in\mathbf{N}^{\ast}}\langle x,x\rangle\{\frac{1}{2}+\frac{1}{i}\}_{i\in\mathbf{N}^{\ast}}.$$ 
So $\{\Lambda_{j}\}_{j}$ is a $\{\frac{1}{2}+\frac{1}{i}\}_{i\in\mathbf{N}^{\ast}}$-tight $\ast$-g-frame.

Let $K:\mathcal{H}\to\mathcal{H}$ defined by $Kx=\{\frac{x_{i}}{i}\}_{i\in\mathbf{N}^{\ast}}$.

Then for every $x\in\mathcal{H}$ we have $$\langle K^{\ast}x,K^{\ast}x\rangle_{\mathcal{A}}\leq\sum_{j\in\mathbf{N}}\langle\Lambda_{j}x,\Lambda_{j}x\rangle=\{\frac{1}{2}+\frac{1}{i}\}_{i\in\mathbf{N}^{\ast}}\langle x,x\rangle\{\frac{1}{2}+\frac{1}{i}\}_{i\in\mathbf{N}^{\ast}}.$$

Now, let $(\Omega, \mu)$ be a $\sigma$-finite measure space with infinite measure and $\{H_{\omega}\}_{\omega \in \Omega}$ be a family of Hilbert $A$-module $(H_{\omega}=C_{0}, \textcolor{white}{..}\forall w \in \Omega)$.\\

Since $\Omega$ is a $\sigma$-finite, it can be written as a disjoint union $\Omega = \bigcup \Omega_{\omega}	$ of countably many subsets $\Omega_{\omega} \subseteq \Omega $, such that $\mu  (\Omega_{k}) < \infty, \textcolor{white}{..} \forall k \in \mathbb {N}$. Without less of generality, assume that $\mu  (\Omega_{k}) > 0 \textcolor{white}{..} \forall k \in \mathbb {N}$.\\
 For each $\omega \in \Omega $, define the operator : $\Lambda_{\omega} : H \rightarrow H_{w}$ by : 
\begin{equation*}
\Lambda_{w}(x) = \frac{1}{\mu  (\Omega_{k})}\langle x,f_{k}\rangle h_{\omega}, \textcolor{white}{..} \forall x \in H
\end{equation*}   
where $k$ is such that $w\in \Omega_{\omega}$ and $h_{\omega}$ is an arbitrary element of $H_{\omega}$, such that $\|h_{\omega}\|= 1$.\\
For each $ x\in H$, $\{\Lambda_{\omega}x\}_{\omega \in \Omega}$ is strongly measurable (since $h_{\omega}$ are fixed) and 
\begin{equation*}
\int_{\Omega}\langle \Lambda_{\omega}x,\Lambda_{\omega}x \rangle d\mu (\omega) = \sum_{j\in \mathbb{N}}\langle x,f_{j}\rangle\langle f_{j},x\rangle
\end{equation*} 
So, therefore 
\begin{align*}
\langle K^{\ast}x,K^{\ast}x\rangle \leq \int_{\Omega}\langle \Lambda_{\omega}x,\Lambda_{\omega}x \rangle d\mu (\omega)&=\sum_{j\in \mathbb{N}}\langle x,f_{j}\rangle\langle f_{j},x\rangle\\
&=\{\frac{1}{2} + \frac{1}{i}\}_{i\in \mathbb {N}^{\ast}}\langle x,x\rangle\{\frac{1}{2} + \frac{1}{i}\}_{i\in \mathbb {N}^{\ast}}
\end{align*}
So $\{\Lambda_{\omega}\}_{\omega \in \Omega}$ is a continuous $\ast$-K-g-frame.
\end{example}
\begin{remark}\textcolor{white}{.}
	\begin{itemize}
	
	\item Every continuous $\ast$-g-frame is a continuous $\ast$-K-g-frame

indeed:
	
	Let $\{\Lambda_{w}\in End_{\mathcal{A}}^{\ast}(U,V_{w}): w\in\Omega\}$ a continuous $\ast$-g-frame for Hilbert $C^{\ast}$-module $U$ with respect to $\{V_{w}: w\in\Omega\}$, then 
	\begin{equation*} 
	A\langle x,x\rangle A^{\ast}\leq\int_{\Omega}\langle\Lambda_{w}x,\Lambda_{w}x\rangle d\mu(w)\leq B\langle x,x\rangle B^{\ast}, \forall x\in U.
	\end{equation*}
	
	or
\begin{equation*} 
\langle K^{\ast}x,K^{\ast}x\rangle \leq||K||^{2}\langle x,x\rangle, \forall x\in U.
\end{equation*}
then

	\begin{equation*} 
	(||K||^{-1}A)\langle K^{\ast}x,K^{\ast}x\rangle(||K||^{-1}A)^{\ast}\leq\int_{\Omega}\langle\Lambda_{w}x,\Lambda_{w}x\rangle d\mu(w)\\
	\leq B\langle x,x\rangle B^{\ast}
	\end{equation*}
	so $\{\Lambda_{w}\in End_{\mathcal{A}}^{\ast}(U,V_{w}): w\in\Omega\}$ be a continuous $\ast$-K-g-frame with lower and upper bounds $||K||^{-1}A$ and $B$, respectively.
\item If $K\in  End_{\mathcal{A}}^{\ast}(H)$ is a surjectif operator, then every continuous $\ast$-K-g-frame for $H$ with respect to $\{V_{w}: w\in\Omega\}$ is a continuous $\ast$-g-frame.

indeed:

 If $K$ is surjectif there exists $m>0$ such that 
\begin{equation*}
m\langle x,x\rangle \leq \langle K^{\ast}x,K^{\ast}x\rangle
\end{equation*}
then
\begin{equation*}
(A\sqrt{m})\langle x,x\rangle (A\sqrt{m})^{\ast} \leq A\langle K^{\ast}x,K^{\ast}x\rangle A^{\ast}
\end{equation*}
or $\{\Lambda_{w}\in End_{\mathcal{A}}^{\ast}(U,V_{w}): w\in\Omega\}$ be a continuous $\ast$-K-g-frame, we have
\begin{equation*}
(A\sqrt{m})\langle x,x\rangle (A\sqrt{m})^{\ast} \leq\int_{\Omega}\langle\Lambda_{w}x,\Lambda_{w}x\rangle d\mu(w)\leq B\langle x,x\rangle B^{\ast}
\end{equation*}
hence  $\{\Lambda_{w}\in End_{\mathcal{A}}^{\ast}(U,V_{w}): w\in\Omega\}$ be a continuous $\ast$-g-frame for $U$ with lower and upper bounds $A\sqrt{m}$ and $B$, respectively
\end{itemize}
\end{remark}
Let $K\in  End_{\mathcal{A}}^{\ast}(U)$,
and $\{\Lambda_{w}\in End_{\mathcal{A}}^{\ast}(U,V_{w}): w\in\Omega\}$ a continuous $\ast$-K-g-frame for Hilbert $C^{\ast}$-module $U$ with respect to $\{V_{w}: w\in\Omega\}$.

We define an operator $T:U\rightarrow\oplus_{w\in\Omega}V_{w}$ by:\\ $Tx=\{\Lambda_{w}x: w\in\Omega\} \forall x\in U$,\\
then $T$ is called the continuous $\ast$-K-g-frame transform.

So its adjoint operator is: $T^{\ast}:\oplus_{w\in\Omega}V_{w}\rightarrow U$ given by:\\ $T^{\ast}(\{x_{\omega}\}_{\omega \in \Omega})=\int_{\Omega}\Lambda^{\ast}_{\omega}x_{\omega} d\mu(w)$\\
By composing $T$ and $T^{\ast}$, the frame operator $S=T^{\ast}T$ given by:\\ $Sx=\int_{\Omega}\Lambda^{\ast}_{\omega}\Lambda_{\omega}x d\mu(w)$,
S is called continuous $\ast$-K-g frame operator
\begin{theorem}\textcolor{white}{.}
	
	The continuous $\ast$-K-g frame operator $S$ is a bounded, positive, selfadjoint and $||A^{-1}||^{-2}\|K\|^{2}\leq ||S||\leq ||B||^{2}$
\end{theorem}
\begin{proof}
	First we show, $S$ is a selfadjoint operator. By definition we have $\forall x, y\in U$
	\begin{align*}
	\langle Sx,y\rangle&=\left\langle\int_{\Omega}\Lambda_{w}^{\ast}\Lambda_{w}xd\mu(w),y\right\rangle\\
	&=\int_{\Omega}\langle\Lambda_{w}^{\ast}\Lambda_{w}x,y\rangle d\mu(w)\\
	&=\int_{\Omega}\langle x,\Lambda_{w}^{\ast}\Lambda_{w}y\rangle d\mu(w)\\
	&=\left\langle x,\int_{\Omega}\Lambda_{w}^{\ast}\Lambda_{w}yd\mu(w)\right\rangle\\
	&=\langle x,Sy\rangle.
	\end{align*}
	Then $S$ is a selfadjoint.
	
	Clearly $S$ is positive.
	
	By definition of a continuous $\ast$-K-g-frame we have
	\begin{equation*}
	A\langle K^{\ast}x,K^{\ast}x\rangle A^{\ast}\leq\int_{\Omega}\langle\Lambda_{w}x,\Lambda_{w}x\rangle d\mu(w)\leq B\langle x,x\rangle B^{\ast}.
	\end{equation*}
	So
	\begin{equation*}
	A\langle K^{\ast}x,K^{\ast}x\rangle A^{\ast}\leq\langle Sx,x\rangle\leq B\langle x,x\rangle B^{\ast}.
	\end{equation*}
	This give 
	\begin{equation*}
	\|A^{-1}\|^{-2}\|\langle KK^{\ast}x,x\rangle\|\leq\|\langle Sx,x\rangle\|\leq \|B\|^{2}\|\langle x,x\rangle\|.
	\end{equation*}
	
	If we take supremum on all $x\in U$, where $\|x\|\leq1$, we have  $$\|A^{-1}\|^{-2}\|K\|^{2}\leq\|S\|\leq\|B\|^{2}.$$
\end{proof}

\begin{theorem}\textcolor{white}{.}
	
	Let $K\in  End_{\mathcal{A}}^{\ast}(H)$  be surjective and $\{\Lambda_{w}\in End_{\mathcal{A}}^{\ast}(U,V_{w}): w\in\Omega\}$ be a continuous $\ast$-K-g-frame for $U$, with lower and upper bounds $A$ and $B$, respectively and with the continuous $\ast$-K-g-frame operator $S$.\\ Let $T\in End_{\mathcal{A}}^{\ast}(U)$ be invertible, then $\{\Lambda_{w}T: w\in\Omega\}$ is a continuous $\ast$-K-g-frame for $U$ with continuous $\ast$-K-g-frame operator $T^{\ast}ST$.
\end{theorem}
\begin{proof}\textcolor{white}{.}
	
	We have 
	\begin{equation}\label{eq11}
	A\langle K^{\ast}Tx,K^{\ast}Tx\rangle A^{\ast}\leq\int_{\Omega}\langle\Lambda_{w}Tx,\Lambda_{w}Tx\rangle d\mu(w)\leq B\langle Tx,Tx\rangle B^{\ast}, \forall x\in U.
	\end{equation}
	Using Lemma (1.3) , we have $\|(T^{\ast}T)^{-1}\|^{-1}\langle x,x\rangle\leq\langle Tx,Tx\rangle$, $\forall x\in U$.\\ 
	 $K$ is surjectif, then there exist $m$ such that: 
	\begin{equation*}
		m\langle Tx,Tx\rangle \leq \langle K^{\ast}Tx,K^{\ast}Tx\rangle
			\end{equation*}
			then
				\begin{equation*}
			 m\|(T^{\ast}T)^{-1}\|^{-1}\langle x,x\rangle\leq \langle K^{\ast}Tx,K^{\ast}Tx\rangle
				\end{equation*}		
			
	so
			\begin{equation*}
			m\|(T^{\ast}T)^{-1}\|^{-1}A\langle x,x\rangle A^{\ast}\leq A\langle K^{\ast}Tx,K^{\ast}Tx\rangle A^{\ast}
			\end{equation*}

	Or $\|T^{-1}\|^{-2}\leq\|(T^{\ast}T)^{-1}\|^{-1}$, this implies:
	
	\begin{equation}\label{eq22} 
	(\|T^{-1}\|^{-1}\sqrt{m}A)\langle x,x\rangle(\|T^{-1}\|^{-1}\sqrt{m}A)^{\ast}\leq A\langle K^{\ast}Tx,K^{\ast}Tx\rangle A^{\ast}, \forall x\in U.
	\end{equation}
	And we know that $\langle Tx,Tx\rangle\leq\|T\|^{2}\langle x,x\rangle$, $\forall x\in U$. This implies that:
	\begin{equation}\label{eq33}
	B\langle Tx,Tx\rangle B^{\ast}\leq(\|T\|B)\langle x,x\rangle(\|T\|B)^{\ast}, \forall x\in U.
	\end{equation}
	Using \eqref{eq11}, \eqref{eq22}, \eqref{eq33} we have:
	\begin{equation*}
	(\|T^{-1}\|^{-1}\sqrt{m}A)\langle x,x\rangle(\|T^{-1}\|^{-1}\sqrt{m}A)^{\ast}\leq \int_{\Omega}\langle\Lambda_{w}Tx,\Lambda_{w}Tx\rangle d\mu(w)\leq (\|T\|B)\langle x,x\rangle(\|T\|B)^{\ast}
	\end{equation*}

	So $\{\Lambda_{w}T: w\in\Omega\}$ is a continuous $\ast$-K-g-frame for $U$.
	
	Moreover for every $x\in U$, we have
	\begin{align*}
	T^{\ast}STx&=T^{\ast}\int_{\Omega}\Lambda_{w}^{\ast}\Lambda_{w}Txd\mu(w)\\&=\int_{\Omega}T^{\ast}\Lambda_{w}^{\ast}\Lambda_{w}Txd\mu(w)\\&=\int_{\Omega}(\Lambda_{w}T)^{\ast}(\Lambda_{w}T)xd\mu(w).
	\end{align*}
	This completes the proof.	
\end{proof}
\begin{corollary}\textcolor{white}{.}
	
Let $\{\Lambda_{w}\in End_{\mathcal{A}}^{\ast}(U,V_{w}): w\in\Omega\}$ be a continuous $\ast$-K-g-frame for $U$ and $K\in  End_{\mathcal{A}}^{\ast}(U)$ be surjective, with continuous $\ast$-K-g-frame operator $S$. Then $\{\Lambda_{w}S^{-1}: w\in\Omega\}$ is a continuous $\ast$-K-g-frame for $U$.
\end{corollary}
\begin{proof}\textcolor{white}{.}
	
Result from the last theorem by taking $T=S^{-1}$	
\end{proof}

the following lemma caracterize a continuous $\ast$-K-g-frame by its frame operator

\begin{theorem}\textcolor{white}{.}
	
	Let $\{\Lambda_{\omega}\}_{\omega \in \Omega}$ be a continuous $\ast$-g-Bessel for $H$ with respect $\{H_{\omega}\}_{\omega \in \Omega}$.
	
	then $\{\Lambda_{\omega}\}_{\omega \in \Omega}$ is a continuous $\ast$-K-g-frame for $H$ with respect to $\{H_{\omega}\}_{\omega \in \Omega}$ if and only if there exist a constant $A>0$ such that $S\geq AKK^{\ast}$ where $S$ is the frame operator for $\{\Lambda_{\omega}\}_{\omega \in \Omega}$. 
\end{theorem}

\begin{proof}\textcolor{white}{.}
	
	We know $\{\Lambda_{\omega}\}_{\omega \in \Omega}$ is a continuous $\ast$-K-g-frame for $H$ with bounded $A$ et $B$ if and only if 
	\begin{equation*}
	A\langle K^{\ast}x,K^{\ast}x\rangle A^{\ast}\leq\int_{\Omega}\langle\Lambda_{w}x,\Lambda_{w}x\rangle d\mu(w)\leq B\langle x,x\rangle B^{\ast}
	\end{equation*}
	If and only if 
	\begin{equation*}
	A\langle KK^{\ast}x,x\rangle A^{\ast}\leq\int_{\Omega}\langle\Lambda^{\ast}_{w}\Lambda_{w}x,x\rangle d\mu(w)\leq B\langle x,x\rangle B^{\ast}
	\end{equation*}
	If and only if 
	\begin{equation*}
	A\langle KK^{\ast}x,x\rangle A^{\ast}\leq\langle Sx,x\rangle \leq B\langle x,x\rangle B^{\ast}
	\end{equation*}
	Where $S$ is the continuous $\ast$-K-g frame operator for  $\{\Lambda_{\omega}\}_{\omega \in \Omega}$.
	
	Therefore, the conclusuin holds.
\end{proof}

\bibliographystyle{amsplain}

\begin{thebibliography}{XX}
	\bibitem{11} S. T. Ali, J. P. Antoine, J. P. Gazeau, \emph{Continuous frames in Hilbert spaces}, Annals of Physics
222 (1993), 1-37.
\bibitem{12} A. Askari-Hemmat, M. A. Dehghan, M. Radjabalipour, \emph{Generalized frames and their redundancy}, Proc. Amer. Math. Soc. 129 (2001), no. 4, 1143-1147.
\bibitem{Ara} L. Aramba\v{s}i\'{c}, \emph{On frames for countably generated Hilbert $\mathcal{C}^{\ast}$-modules}, Proc. Amer. Math. Soc. 135 (2007) 469-478.
\bibitem{Deh} A.Alijani,M.Dehghan, \emph{$\ast$-frames in Hilbert $\mathcal{C}^{\ast}$modules},U. P. B. Sci. Bull. Series A 2011.
\bibitem{Con} J.B.Conway ,\emph{A Course In Operator Theory},AMS,V.21,2000.
\bibitem{13} I. Daubechies, A. Grossmann, and Y. Meyer, \emph{Painless nonorthogonal expansions}, J. Math.
Phys. 27 (1986), 1271-1283.
\bibitem{Duf} R. J. Duffin, A. C. Schaeffer, \emph{A class of nonharmonic fourier series}, Trans. Amer. Math. Soc. 72 (1952),
341-366.
\bibitem{Dav} F. R. Davidson, \emph{$\mathcal{C}^{\ast}$-algebra by example},Fields Ins. Monog. 1996.
\bibitem{Gab} D. Gabor, \emph{Theory of communications}, J. Elec. Eng. 93 (1946), 429-457.
\bibitem{14} J. P. Gabardo and D. Han, \emph{Frames associated with measurable space}, Adv. Comp. Math. 18
(2003), no. 3, 127-147.
\bibitem{15} G. Kaiser, \emph{A Friendly Guide to Wavelets}, Birkha”user, Boston, 1994.
\bibitem{Ross} M. Rossafi, S. Kabbaj, \emph{$\ast$-K-g-frames in Hilbert $C^{\ast}$-modules}, Journal of
Linear and Topological Algebra, Vol. 07, No. 01, 2018, 63-71.
\bibitem{Pas} W. Paschke, \emph{Inner product modules over $B^{\ast}$-algebras}, Trans. Amer. Math. Soc., (182)(1973), 443-468.
 
\end{thebibliography}

\vspace{0.1in}

\end{document}